\documentclass[fleqn,a4paper,12pt]{article}
\usepackage[T2A]{fontenc}
\usepackage[cp1251]{inputenc}
\usepackage[english]{babel}
\usepackage[dvips]{graphicx}
\usepackage{color}
\usepackage{hyperref}
\usepackage{alltt}
\usepackage{afterpage}
\usepackage{tikz}
\usepackage{caption}
\usepackage{subcaption}
\usepackage{graphicx}
\usepackage{amsmath,amssymb,amsthm}
\usepackage{wrapfig}

\newtheorem{lemma}{Lemma}
\newtheorem{conjecture}{Conjecture}[section]
\newtheorem{thm}{Theorem}
\newtheorem{cor}{Corollary}[section]
\newtheorem{prop}{Proposition}[section]
\newcounter{def}

\newcommand{\dist}{\mathop{\rm dist}\nolimits}
\newcommand{\cg}{\mathop{\rm CG}\nolimits}

\newcommand{\eps}{\varepsilon}
\newcommand{\dd}{\mathrm{d}}
\begin{document}

\title{Optimal packings of congruent circles on a square flat torus }
\author{Oleg~R.~Musin\thanks{The author is supported by the Delone Laboratory of Discrete and Computational Geometry (P.~G.~Demidov Yaroslavl State University) under RF Government grant 11.G34.31.0053, RFBR  grant 11-01-00735 and NSF grant DMS-1101688.} and Anton~V.~Nikitenko\thanks{The author is supported by the Chebyshev Laboratory  (Department of Mathematics and Mechanics, St. Petersburg State University) under RF Government grant 11.G34.31.0026 and by the Delone Laboratory of Discrete and Computational Geometry (P.~G.~Demidov Yaroslavl State University) under RF Government grant 11.G34.31.0053}.}

\maketitle
\begin{abstract}
We consider packings of congruent circles on a square flat torus, i.e., periodic (w.r.t. a square lattice) planar circle packings, with the maximal circle radius. This problem is interesting due to a practical reason --- the problem of ``super resolution of images.'' We have found optimal arrangements for $N=6$, $7$ and $8$ circles. Surprisingly, for the case $N=7$ there are three different optimal arrangements. Our proof is based on a computer enumeration of toroidal irreducible contact graphs. 
\end{abstract}

\section{Preface}

\subsection{History}

A class of problems of an optimal (in different senses) placing of equal balls in different containers is a huge class of interesting research problems in discrete geometry. The best known variations are the following: to find the densest packing of equal balls in $\mathbb R^n$, to find the minimal covering of $\mathbb R^n$ by equal balls, to find the maximal number of equal balls that can touch another one with the same radius, and to find the maximal number of equal balls that can fit into a prescribed container. One can also allow balls of different radii or even replace the balls with something like simplices or other objects, or study the universal properties of packings of convex bodies, giving rise to many other similar problems that are also of interest. These problems have been widely studied and appeared to be very complex despite having very simple formulations. Another point is that many of these problems have practical value (even for the ``strange'' containers and higher dimensions). Some of the results are collected in the following books: \cite{BMP}, \cite{Conw}, \cite{FeT}.

In this paper we are talking about one of the aforementioned problems, where the container is the flat torus $\mathbb R^2/\mathbb Z^2$ with the induced metric and the goal is to find the maximal radius of the disk, whose $N$ congruent copies can be fit in the torus without overlapping. The problem is equivalent to finding the densest packing, as the area of the torus is constant and equals one, and so the best density is $$\rho = N\pi R_{max}^2.$$

We found out about this problem from Daniel Usikov (\cite{Us}), who described the practical reason for studying optimal disk packings on the torus. According to him, this case is especially interesting for the problem of ``super resolution of images.'' Here are several extracts from his letter that provide a more detailed description:

\begin{quotation}{
\it
Let's assume that the image obtained by an observation has an infinite accuracy. Now assume that we have just two  observations. If both of them are conducted from the same point, obviously, these two observations are not better than just one.

Now let's shift the position of the second with respect to the first one.  Then any shift would be equivalent from the point of view of super resolution restoration. But in practical applications we need to take into account two sources of observational errors. First, the exact amount of the shift can be found with some accuracy only, and second, the brightness in each pixel of the image has an intrinsic error -- so called ``photon noise''.  So if the shift is within its statistical error, any attempt to calculate the super-resolution image will fail. Furthermore, it is visible from the formulas that they operate with the differences in brightness between adjacent pixels. So, by separating the observation points as far as possible, one can reduce the error in the estimation of the brightness differences.

It definitely looks like the best result of super resolution observation would be to maximize the distance between two points of observations. For the case of the square pixel, it is the half of the pixel's diagonal. When a third point is added, we expect that the best allocation would be to maximize the minimal distance between three point, etc.

I think that a rigorous prove of the max-min requirement is achievable. As a simplification on the road to the prove, one can assume that the brightness in each pixel takes a random value within an interval (say 0:1). Then averaging over all possible random brightness samples, the brightness should gone from the optimization formulas, and the task will be reduced to the search of the optimal location of observation points on the torus.}

\medskip 
 
{\it Sent by email, from Daniel Usikov (dusikov@yahoo.com)

November 22, 2010}
\end{quotation}

Obviously, the optimal packing in the torus could not be worse than the optimal packing in the square $[0,1]^2$. Here are some results for the small number of disks in the square, borrowed from \cite{BMP} (See \cite{Mel}, \cite{Sch}, \cite{SchM}). Here $d$ denotes the distance between the centers. Corresponding configurations are shown in Figure \ref{fig:square}.
\begin{table}[h]
	\centering
		\begin{tabular}{c|c|c|c|c|c|c|c|c|c}
			N & 2 & 3 & 4 & 5 & 6 & 7 & 8 & 9\\
			\hline
			$d \approx$ & 0.5858 & 0.5087 & 0.5000 & 0.4142 & 0.3754 & 0.3489 & 0.3411 & 0.3333
		\end{tabular}
\end{table}
\begin{figure}[h!]
    \centering
    \includegraphics[scale=0.5, clip, trim = 0 0 0 80]{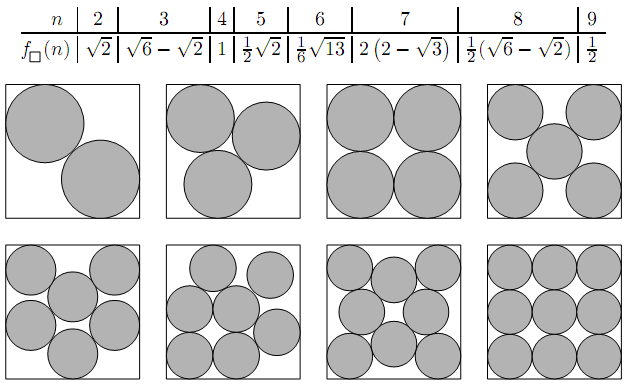}
    \caption{The optimal configurations for the square}
    \label{fig:square}
\end{figure}%

\subsection{The formal statement}

The problem is: for a given number $N\geq 1$ of points, find the maximal $r\in\mathbb{R^+}$ such that $N$ circles of radius $r$ could be put on the square flat torus $\mathbb{T} = \mathbb{R}^2/\mathbb{Z}^2$ without overlapping, or, equivalently, to find the maximal $d\in\mathbb{R^+}$ such that there are $N$ points on the torus with pairwise distances not less than $d$ (where $d=2r$).

We will work with the second statement.

\subsection{Known results for the torus}

Optimal configurations of up to 4 circles are relatively easy to find, and for $N\leq 5$ they were found in \cite{Dick}. These configurations and the corresponding $d$-s are presented in Figure~\ref{fig:optsmall}. D.~Usikov (\cite{Us}) has conjectured some optimal configurations using numerical simulations for the small numbers of disks.
\setlength\fboxsep{0pt}
\setlength\fboxrule{0.5pt}
\begin{figure}[tb]%
	 \centering
        \begin{subfigure}[b]{0.25\textwidth}
                \centering
                \fbox{\includegraphics[clip, trim = 1 1 3 3,scale=0.15]{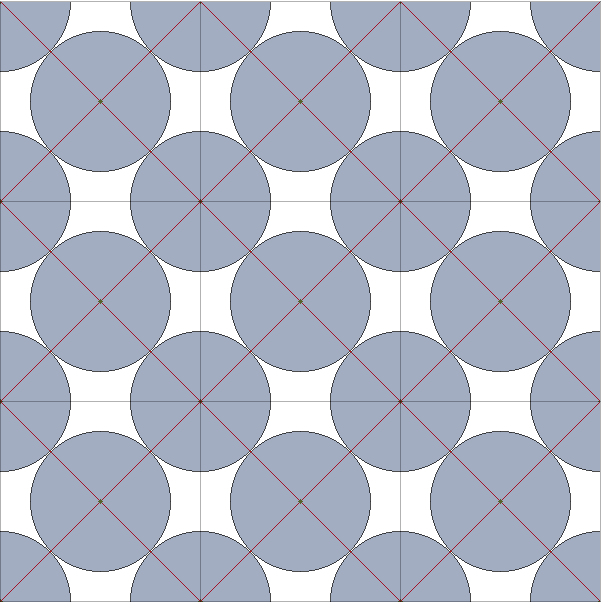}}
                \caption[]{$\begin{array}{c}N=2\\ d=\frac{\sqrt{2}}{2} \end{array}$}
                \label{fig:opt2}
        \end{subfigure}%
        \begin{subfigure}[b]{0.25\textwidth}
                \centering
                \fbox{\includegraphics[clip, trim = 1 1 3 3,scale=0.15]{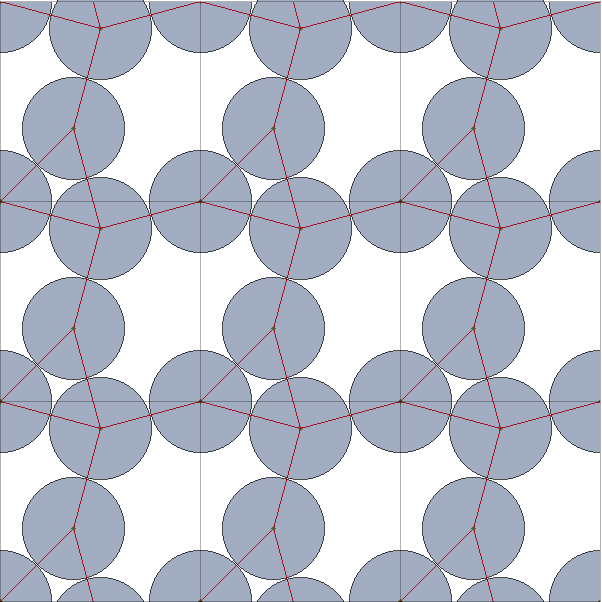}}
                \caption[]{$\begin{array}{c}N=3\\ d=\frac{\sqrt{6}-\sqrt{2}}{2} \end{array}$}
                \label{fig:opt3}
        \end{subfigure}%
        \begin{subfigure}[b]{0.25\textwidth}
                \centering
                \fbox{\includegraphics[clip, trim = 1 1 3 3,scale=0.15]{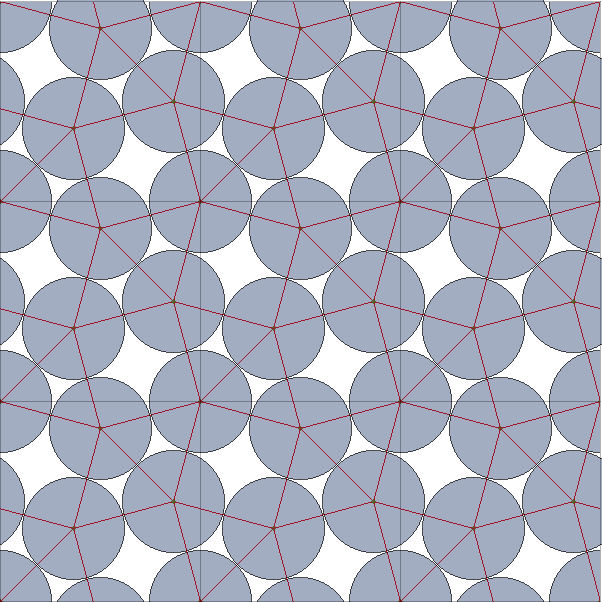}}
                \caption[]{$\begin{array}{c}N=4\\ d=\frac{\sqrt{6}-\sqrt{2}}{2} \end{array}$}
                \label{fig:opt4}
        \end{subfigure}%
        \begin{subfigure}[b]{0.25\textwidth}
                \centering
                \fbox{\includegraphics[clip, trim = 1 1 3 3,scale=0.15]{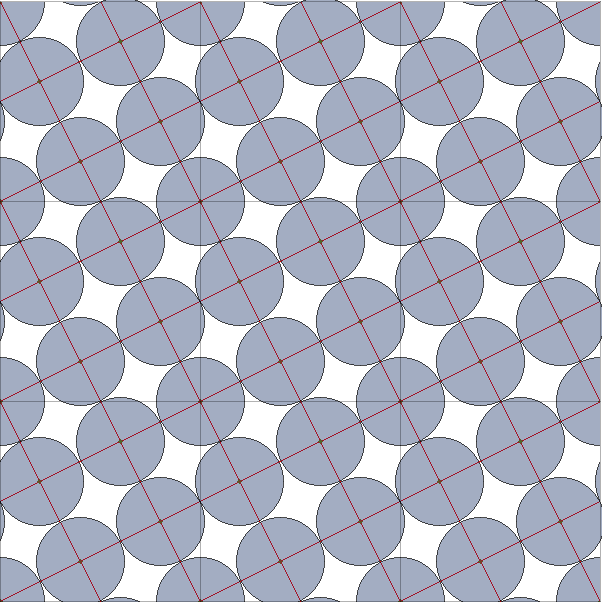}}
                \caption[]{$\begin{array}{c}N=5\\ d=\frac{\sqrt{5}}{5} \end{array}$}
                \label{fig:opt5}
        \end{subfigure}%
    \caption{Optimal configurations of up to five points in the torus}
    \label{fig:optsmall}
\end{figure}

\section{Main results}

In this paper, we present solutions of the problem for $N=6, 7$ and $8$ and a conjecture for the case $N=9$.

\begin{thm} 
\label{thm6}
The arrangement of 6 points in ${\mathbb T}^2$ which is shown in Figure \ref{fig:opt6}
is the best possible, the maximal arrangement is unique up to isometry, and $d \approx 0.40040554$. The precise formula is
$$d=\frac{\sqrt{3}}{2} - \frac{\sqrt{2}\sqrt{3\sqrt{3} + 2}}{6} + \frac 16.$$
\end{thm}

\begin{thm} 
\label{thm7}
There are three, up to isometry or up to a move of a free disk, optimal arrangements of 7 points in ${\mathbb T}^2$ which are shown in Figures \ref{fig:opt7_1} -- \ref{fig:opt7_3} where $d = \frac{1}{1+\sqrt{3}}\approx0.3660$.
\end{thm}

\begin{thm} 
\label{thm8}
There is one unique, up to isometry, optimal arrangement of 8 points in ${\mathbb T}^2$, which is shown in Figure \ref{fig:opt8} where $d = \frac{1}{1+\sqrt{3}}\approx0.3660$.
\end{thm}

\begin{conjecture} 
\label{conj9}
There is one unique, up to isometry, optimal arrangement of 9 points in ${\mathbb T}^2$, which is shown in the Figure \ref{fig:opt9} where $d = \frac{1}{\sqrt{5+2\sqrt{3}}}\approx0.3437$.
\end{conjecture}

So, for $N=6, 7$ and $8$ we have a $5-7$\%--improvement compared to the results for the square.  The exact point coordinates can be found in the Appendix \ref{coords}.
\afterpage
{
	\begin{figure}[t]
	    \centering
	        \begin{subfigure}[b]{0.45\textwidth}
	                \centering
							    \fbox{\includegraphics[clip, trim = 1 1 3 3,scale=0.3]{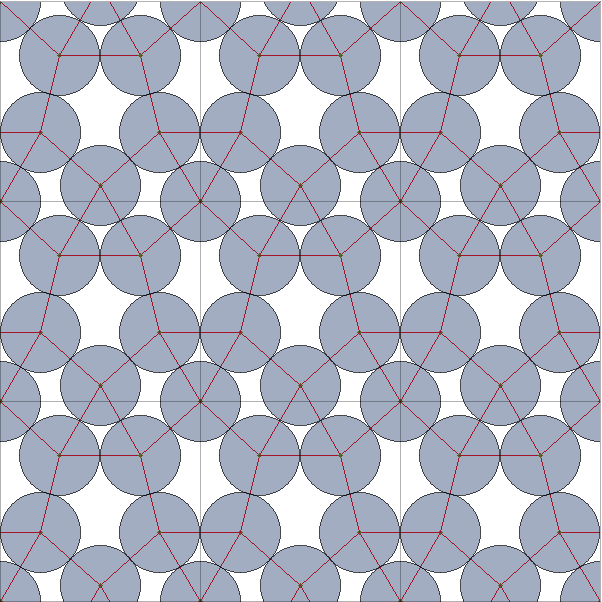}}
							    \caption[]{$N=6$, optimal}
	                \label{fig:opt6}
	        \end{subfigure}%
	        \begin{subfigure}[b]{0.45\textwidth}
	                \centering
	                \fbox{\includegraphics[clip, trim = 1 1 3 3,scale=0.3]{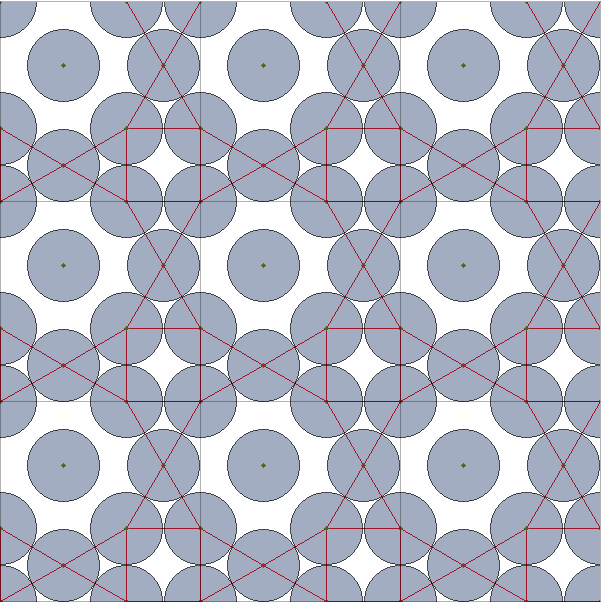}}
	                \caption{$N=7$, optimal}
	                \label{fig:opt7_1}
	        \end{subfigure}%
	        
	        \begin{subfigure}[b]{0.45\textwidth}
	                \centering
	                \fbox{\includegraphics[clip, trim = 1 1 3 3,scale=0.3]{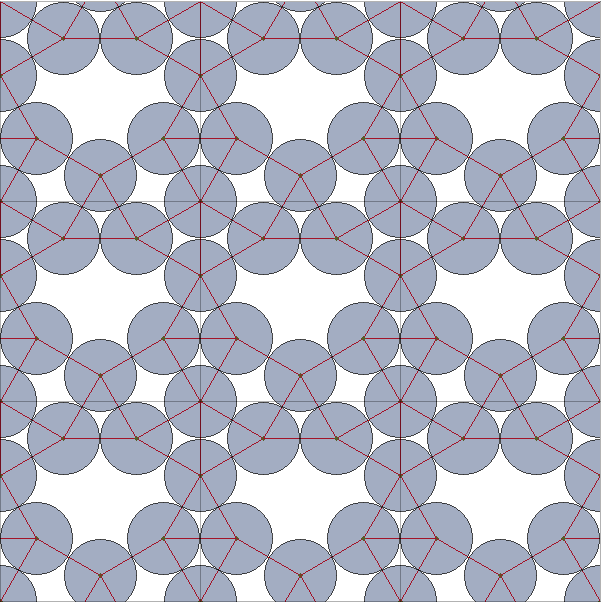}}
	                \caption{$N=7$, optimal}
	                \label{fig:opt7_2}
	        \end{subfigure}%
	        \begin{subfigure}[b]{0.45\textwidth}
	                \centering
	                \fbox{\includegraphics[clip, trim = 1 1 3 3,scale=0.3]{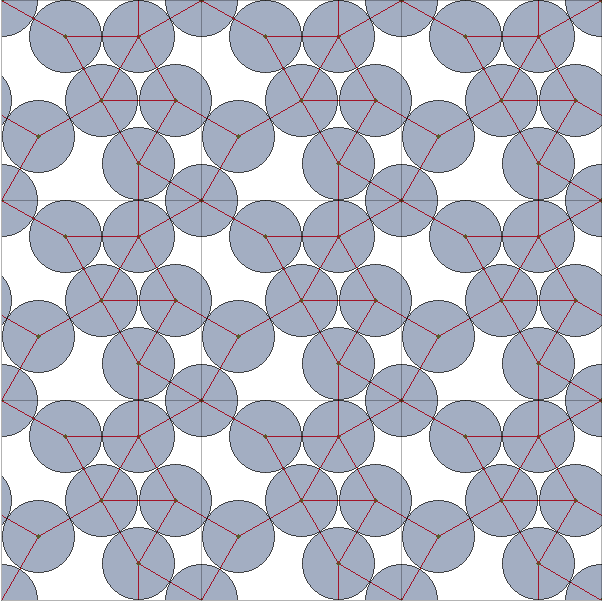}}
	                \caption{$N=7$, optimal}
	                \label{fig:opt7_3}
	        \end{subfigure}%
	        
	        \begin{subfigure}[b]{0.45\textwidth}
	                \centering
							    \fbox{\includegraphics[clip, trim = 1 1 3 3,scale=0.3]{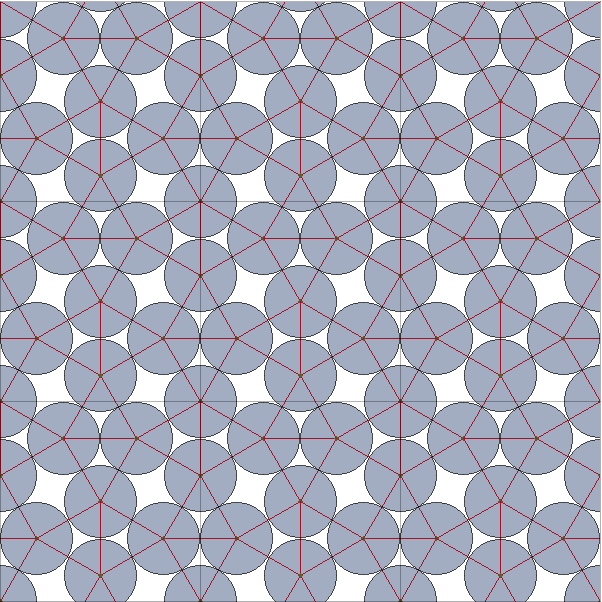}}
							    \caption{$N=8$, optimal}
							    \label{fig:opt8}
	        \end{subfigure}%
	        \begin{subfigure}[b]{0.45\textwidth}
	                \centering
							    \fbox{\includegraphics[clip, trim = 1 1 3 3,scale=0.3]{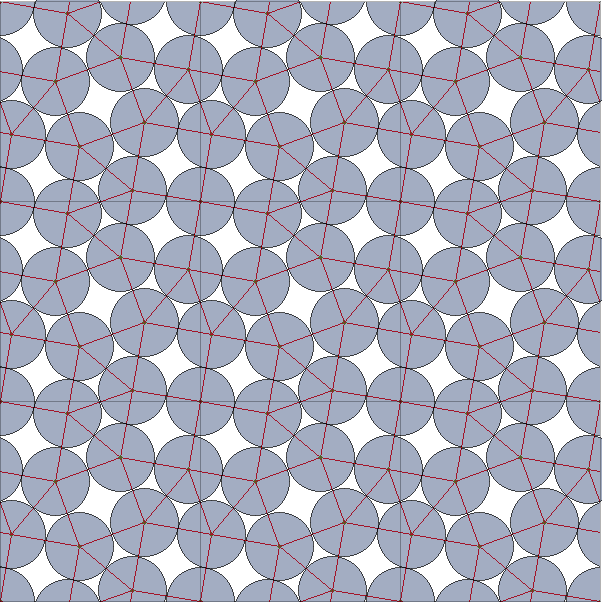}}
							    \caption{$N=9$, conjecture}
							    \label{fig:opt9}
	        \end{subfigure}%
	        \caption{Configurations for $N=6, 7, 8$ and $N=9$}
	        \label{fig:optima}
	\end{figure}%
	\clearpage
}

\section{Basic definitions and preliminary results}

The main notion we use to find an optimal configuration is the notion of a contact graph.

\medskip

\noindent{\bf Contact graphs.} Let $X$ be a finite set in $\mathbb
T^2$. The {\it contact graph} $\cg(X)$  is an embedded graph with vertices $X$ and geodesic edges $(x,y), \, x,y\in X$ such that $x$ and $y$ are adjacent iff $\dist(x,y) = \min\limits_{u,v\in X}\dist(u,v)$.

As a corollary of the definition, one sees that all the edges in a contact graph have equal lengths.

\medskip

We refer to the configuration of $N$ points that maximizes the minimal distance between the $N$ points, i.e., solves the packing problem, as {\it optimal}. To avoid dealing with an infinite family of torus isometries induced by parallel transports of the covering plane, from now on we suppose that every point set in question contains the vertex $(0,0)$. 

\medskip

\noindent{\bf Irreducible contact graphs.}
We say that a contact graph $\cg(X)$ (and its corresponding configuration $X$) is {\it irreducible} if it is rigid in a sense that there are no possible slight motions of vertices (we do not distinguish the points in $X$ and the vertices of the graph) that change the contact graph (either the graph itself or the topology of its embedding) or increase the edge lengths. A more detailed definition of the contact graph for the torus using so called {\it strut frameworks} and examples can be found in \cite{Dick} (there it is called ``infinitesimaly rigid strut framework''), which as they state is equivalent in our case to the definition in \cite{Con1}. The only difference between our definitions is that we allow an irreducible contact graph to have isolated vertices.

This definition can be applied to point configurations in any metric space. Note that W. Habicht, K. Sch\"utte,  B.L. van der Waerden, and L. Danzer  used irreducible contact graphs for the kissing number and Tammes problems for $\mathbb S^2$ \cite{HabvdW, SvdW1, SvdW2, Dan}. The idea of irreducibility is the expression of the thought that the optimal configuration should have some nice properties and there is a possibility to ``adjust'' any optimal configuration to have even more nice properties. Here we only want the configuration to be rigid, but, for example, Danzer \cite{Dan} used a trick to disallow the optimal graph to admit another adjustment that is now called ``Danzer flip'', and that helped him to solve the Tammes problem for 10 and 11 points. Also this trick helped to kill lots of graphs in \cite{Mus}.

Now we state some basic propositions that will be of significant use.

\begin{prop} Let $X$ be an optimal (corr. locally optimal, i.e., giving the local maximum of $d$ in the natural topology) configuration of $N$ points on a square flat torus. Then its contact graph is either irreducible, or $X$ can be continuously moved to another optimal (corr. locally optimal) configuration with an irreducible contact graph.
\end{prop}

Although this fact is easy to prove, we refer to \cite{Con1}. Also see Theorem 3.2 in \cite{Dick}.

\begin{prop} For any configuration $X$ its contact graph can have vertices of degree at most six.
\end{prop}
\begin{proof}It is obvious since any circle in the plane can not touch more than 6 circles of the same radius, and since the universal cover of the contact graph is the contact graph in the plane, the property holds.\end{proof}

\begin{prop} For any configuration $X$ its contact graph can not have angles between adjacent edges below $\frac\pi3$.
\end{prop}
\begin{proof}It is obvious since an angle must be not less than in an equilateral triangle.\end{proof}

\begin{prop} There are no angles greater or equal than $\pi$ in an irreducible configuration.\end{prop}
\begin{proof}If there is an angle of $\pi$ or greater, than the corresponding vertex can be slightly ``pushed'' inside a face with that angle, and it becomes isolated.\end{proof}

\begin{cor} There are no vertices of degree one or two in an irreducible configuration.\end{cor}

\begin{prop} An irreducible contact graph without isolated vertices is cellular (i.e., the interior of each face is homeomorphic to an open disk).\end{prop}
\begin{proof} If not, then one can easily see that such graph has an obtuse or a flat angle. For the details see \cite{Dick}.\end{proof}

\begin{prop} An isolated vertex in the contact graph can be only inside a face, whose universal cover is a polygon of at least seven vertices. (The universal cover is referred to for the strictness, as there can be, for example, ``octagonal'' faces on the torus with only 4 distinct vertices. Nevertheless, in the following we will call such faces just ``octagons''.)
\end{prop}
\begin{proof}If we draw all straight-line segments from the isolated vertex $v$ to all the vertices of the face, all of the segments will be strictly longer than the edges, thus the angles are strictly less than $\frac\pi3$. Their sum is $2\pi$, so there must be at least 7 such angles. By the way, if there is a vertex inside a hexagon, then it is a regular hexagon split into 6 equilateral triangles.\end{proof}

{\noindent \bf Remark}. If we remove an isolated vertex from the contact graph on $N$ vertices, then we will get a locally optimal configuration on $N-1$ vertices.

\begin{prop} \label{edgebnd} An irreducible contact graph without isolated vertices has at least $2N-1$ edges.\end{prop}
\begin{proof} We refer to the result of Connely \cite{Con2}, \cite{Con3}. Some general position ideas that can be used to prove this will be given in the next sections.\end{proof}

Now we state some propositions involving the $d$ we are looking for.

\begin{prop} Let $d(N)$ denote a solution to the packing problem for $N$ points. Then $d(N)$ is decreasing.\end{prop}
\begin{proof} One can remove any vertex to achieve a configuration on $N-1$ points.\end{proof}

\begin{prop} The following inequality holds: $d(N)\leq\dfrac{2}{\sqrt{N\sqrt{12}}}$.\end{prop}
\begin{proof} It follows from the knowledge of the maximal density of circle packing on a plane, which is $\dfrac{\pi}{\sqrt{12}}$ (\cite{Fejes}, \cite{FeT}). See \cite{Dick}.\end{proof}

\begin{prop} There are no multiple edges in the contact graph for $N\geq 5$ points.\end{prop}
\begin{proof} There are no multiple edges in the universal cover since any two disks on the plane cannot touch in two points. So the multiple edges $e_1=(uv)$ and $e_2=(vu)$ may occur only if the cycle $uvu = e_1e_2e_1$ wraps around the torus (i.e., is not homologically trivial). But then the length of the edge must be at least $\frac 12$, which is not possible as $d(5)=\frac{\sqrt{5}}{5}<\frac 12$ (\cite{Dick}) and $d(N)$ decreases.\end{proof}

\section{The idea of the proof}

Now we have some basic properties of contact graphs, and so we are going to use a computer to find the one that will give the optimal configuration, among all the graphs with such properties; the idea is similar to the one used for solving the Tammes problem in \cite{Mus}. There, to find an optimal configuration on the sphere, all the planar graphs with some necessary properties have been enumerated and tried to be embedded into a sphere with equal edge lengths, satisfying the properties of contact graphs. The enumeration has been done by a program {\ttfamily plantri} by G.~Brinkmann and B.~McKay(\cite{plantri}).

We've used the program {\ttfamily surftri} written by T.~Sulanke (\cite{surftri}), who very kindly modified the original release, that enumerated only the triangulations, to be able to enumerate all toroidal cellular graphs, and we are very grateful to him for that.

But unlike \cite{Mus}, where all planar graphs have been enumerated and the possibility of irreducible embedding has been checked, we've used a slightly different approach. The idea grows from the fact that since we are unable to find an embedding of the graph exactly, in any case it would be very hard to distinguish the graph and its subgraphs programmatically. Here are the details.

Suppose we have an embedding of the graph $G$ without isolated vertices, like in any of the figures (except for Figure \ref{fig:opt7_1}) above. Then, considering each edge as a vector, we may write the following system of equations that characterizes the embedding (up to a parallel transport that is out of interest for us):
\begin{equation}
\label{syst}
\left\{\begin{array}{ll}
|v_i|^2 = d^2 & i = 1..E \\
\sum\limits_{v_i \text{ is an edge of } \mathrm{f}}\,\,\pm v_i = 0 & \mathrm{f}\text{ is a face of } G \\
\sum\limits_{v_i \text{ is an edge of } c_1} \pm v_i = \mathrm{const}_1 & \\
\sum\limits_{v_i \text{ is an edge of } c_2} \pm v_i = \mathrm{const}_2. &
\end{array}\right.
\end{equation}
Here we assign a vector $v_i$ to each edge of $G$, fixing any of the two possible directions of the edge. Then the first $E$ equations tell that all edge lengths are equal to $d$. The next $F$ equations say that each face is really a face, i.e., the sum of the vectors corresponding to its edges (with appropriate signs) equals to zero. The last two equations ``fix the cut'' of the torus: we choose any two cycles $c_1$ and $c_2$ that do not form a face (i.e., homologically independent) and prescribe for each of them the way it wraps around the torus. Since it is a cycle, then the sum of the vectors, corresponding to its edges, has integer coordinates, and based on the number of the edges in the cycle, we easily establish that there is only a finite number of possible values of $\mathrm{const}_1$ and $\mathrm{const}_2$ if the cycle lengths are bounded. The details will follow in lemmas \ref{lc1}--\ref{lc4}. An example of the system for a specific graph is System (\ref{syststb}) on page \pageref{syststb}.

\noindent{\bf Remark.} It is well-known that a graph embedded on a surface is unique up to orientation-preserving homeomorphism of the surface characterized by its rotation system --- a cyclic ``clockwise'' order in each vertex of the out-going edges. And {\ttfamily surftri} outputs these rotation systems. However, there are many homeomorphisms of the torus from $\mathrm{SL}_2(\mathbb Z)$ that do not conserve lengths, and so we have to somehow ``fix the homeomorphism'', and in the previous system it is done by selecting the two exclusive cycles and fixing them. The details on that will be below.

Note that the equations on the faces are not linearly independent (except the case where there is only one face, which is trivially impossible with our restrictions on degrees of vertices and without multiple edges): each edge is counted twice, once in positive direction and once in the negative, so the sum of all these equations is zero. Removing one of the equations for the faces, and, rewriting the system in coordinates, we get a system of $m$ equations in $n=2E+1$ variables (coordinates and $d$), where $$m=E+2(F-1)+2\cdot2=E+2F+2.$$ By the Euler formula for a torus, we have a relation $N+F=E$, so $$m=E+2(E-N)+2=3E-2N+2.$$ If a solution exists, by the reasons of general position, we may hope to be able to find it by only considering only $n$ of $m$ equations. On the geometrical level, it may be the same as considering the subgraph instead of the original, hoping to get the removed edges automatically. Taking $m=n$ for a subgraph we find that $E=2N-1$. By Proposition \ref{edgebnd}, the contact graph has at least such number of edges, so there exists its subgraph having exactly $2N-1$ edges\footnote{By the way, the same reasons of general position help us think that if there are less then $n$ equations, i.e., less than $2N-1$ edges, then the space of solutions has positive dimension. Thus there is an infinitesimal motion possible and a graph with less than $2N-1$ edges could not be an irreducible contact graph.}. So, our idea is to consider only the graphs with $2N-1$ edges and to find their embeddings.

\noindent{\bf Remark.} Although the system of equations is polynomial and there is much known about them, algebraic surfaces given by these equations seem not to be in the general position in the whole space, and such a system usually has infinite number of (complex) solutions, so the known approaches for polynomial systems fail.

\section{Technical details}

Now we need some statements to guarantee that the code for solving the problem works.

\noindent{\bf Notation:} For an (oriented) cycle $c$ we write $l(c) := \sum e_i$, where $e_i$ are the vectors $v_i$ from System (\ref{syst}), corresponding to the edges of the cycle, with appropriate signs (that are determined by the orientation).

\noindent{\bf Remark.} In the following, we refer to a pair of (oriented) cycles, a cut along which unwraps the torus to a flat polygon, as ``homologically independent cycles'', and to a cycle, that forms a face, as ``trivial cycle''.

\begin{lemma}\label{lc1} If $d\leq d(5) = \frac {\sqrt{5}}{5}$, then (in an irreducible contact graph) there is no non-trivial cycle of three edges where one edge occurs twice. Also, for any non-trivial cycle $c$ of three edges $l(c) = (\pm 1,0)$ or $l(c)=(0,\pm1)$, and for any non-trivial cycle $s$ of four edges $l(s) \in \{(\pm 1,0),(0,\pm 1),(1,\pm 1),(-1, \pm 1)\}$.\end{lemma}
\begin{proof} It easy follows from such bound on $d$ that these cycles lifted to $\mathbb R^2$ just couldn't reach any other point in $\mathbb Z^2$: the distance in $\mathbb R^2$ between the endpoints of any chain of three edges is not more than $3d$, which is strictly less than $\sqrt{2}$ --- the length of a diagonal of the square, so, if this chain connects the two different lifts of the same vertex, then they could only be different by the unit horizontal or vertical vector. For a cycle of length 4, the argument is the same: the distance is not more than $4d<2$, so the shift is only possible by the unit horizontal or vertical vector or by the diagonal of the square with the length $\sqrt{2}$.\end{proof}

\begin{lemma}\label{lc2} If $d\leq\frac {\sqrt{5}}{5}$ and $c_1$ and $c_2$ are homologically independent cycles (in an irreducible contact graph) having lengths 3 and 3, then there exists an isometry of the torus (possibly orientation-reversing) that makes $l(c_1) = (1, 0)$ and $l(c_2) = (0, 1)$.\end{lemma}
\begin{proof} These sums are originally $\pm (1, 0)$ and $\pm (0, 1)$ or vice versa (by the previous lemma), since the reflections against the lines $y=0$, $y=x$ and $x=0$ are isometries of the torus, we easily achieve what we want.\end{proof}

The next two lemmas have the same proof:

\begin{lemma}\label{lc3} If $d\leq\frac {\sqrt{5}}{5}$ and $c_1$ and $c_2$ are homologically independent cycles (in an irreducible contact graph) having lengths 3 and 4, then there exists an isometry of the torus (possibly orientation-reversing) that makes $l(c_1) = (1, 0)$ and $l(c_2) = (0, 1)$, or $l(c_1) = (1, 0)$ and $l(c_2) = (1, 1)$, or  $l(c_1) = (1, 0)$ and $l(c_2) = (-1, 1)$.\end{lemma}

\begin{lemma}\label{lc4} If $d\leq\frac {\sqrt{5}}{5}$ and $c_1$ and $c_2$ are homologically independent cycles (in an irreducible contact graph) having lengths 4 and 4, then there exists an isometry of the torus (possibly orientation-reversing) that makes $l(c_1) = (1, 0)$ and $l(c_2) = (0, 1)$, or $l(c_1) = (1, 0)$ and $l(c_2) = (1, 1)$, or $l(c_1) = (1, 0)$ and $l(c_2) = (-1, 1)$, or $l(c_1) = (1, -1)$ and $l(c_2) = (1, 1)$.\end{lemma}

The next technical lemma guarantees some nice properties to the subgraph.

\begin{lemma}\label{lm} If $N=6$ or $N=7$, then for each irreducible (locally) optimal contact graph without isolated vertices there exists its subgraph with $2N-1$ edges which is cellular and all its vertices are of degree at least 3.\end{lemma}
\begin{proof} We prove that if there is a cellular graph with degrees of all vertices at least 3 and at least $2N$ edges, then we can remove an edge so that the property holds.

Firstly, we prove the estimate on the number of vertices of degree 3.  Since $$3n_3 + 4n_4 + 5n_5 + 6n_6 = 2E,$$ where $n_k$ is the number of vertices of degree $k$, and $$E \geq 2N = 2n_3 + 2n_4 + 2n_5 + 2n_6,$$ we have $n_3 \leq n_5 + 2n_6$.

Now we prove that there are at most two edges, the removal of which violates cellularity. Suppose there is such an edge $e$. Then there we have one of the following pictures (topologically):

\begin{figure}[h]
	\centering
	\begin{subfigure}[t]{0.28\textwidth}
		\centering
		\resizebox{\textwidth}{130pt}
		{
			\begin{tikzpicture}[scale=0.4]
			
			\draw  (-5,5) rectangle (5,-5);
			\node (v1) at (-3,5) {};
			\node (v2) at (-2.1,0) {};
			\node at (-3,-5) {};
			\node (v7) at (-4,0) {};
			\draw  (-3,0) ellipse (1 and 5);
			\node (v3) at (-0.9,0) {};
			\node (v4) at (0.9,0) {};
			\node at (0,5) {};
			\node at (0,-5) {};
			\draw  (v2) edge node[auto] {e} (v3);
			\node (v5) at (2.1,0) {};
			\node (v8) at (3,0) {};
			\draw  (v4) edge (v5);
			\draw  (0,0) ellipse (1 and 5);
			\draw  (3,-1) rectangle (2,1);
			\node (v9) at (5,0) {};
			\node (v6) at (-5,0) {};
			\draw  (v6) edge (v7);
			\draw  (v8) edge (v9);
			\draw[thick, dashed,] (0,0) -- (0.8,3);
			\draw[thick, dashed,] (0,0) -- (0.8,-3);
			\draw[thick, dashed,] (0.4,1.5) -- (-0.8,3);
			\draw[thick, dashed,] (0,0) -- (-0.95,-2);
			\draw[thick, dashed,] (-0.48,-1) -- (-0.95,2);
			\draw[thick, dashed,] (-0.95,-2) -- (0.8,-3);
			\draw[thick, dashed,] (-3,0) -- (-2.2,3);
			\draw[thick, dashed,] (-3,0) -- (-2.2,-3);
			\draw[thick, dashed,] (-2.6,1.5) -- (-3.8,3);
			\draw[thick, dashed,] (-3,0) -- (-3.95,-2);
			\draw[thick, dashed,] (-3.48,-1) -- (-3.95,2);
			\draw[thick, dashed,] (-3.95,-2) -- (-2.2,-3);
			\draw[thick, dashed,] (2,-0.5) -- (2.4,0.3) -- (2.7,-0.4) -- (3,0);
			\draw[thick, dashed,] (2,0.5) -- (2.4,0.3);
			\draw[thick, dashed,] (2,-0.5) -- (2.7,-0.4);
			
			\end{tikzpicture}
		}
		\caption{2 vertical\\1 ship}
		\label{fig:lem1}
	\end{subfigure}
	\begin{subfigure}[t]{0.28\textwidth}
		\centering
		\resizebox{\textwidth}{130pt}
		{
			\begin{tikzpicture}[scale=0.4]
				\draw  (-5,5) rectangle (5,-5);
				\node (v1) at (-5.9,5) {};
				\node (v2) at (-5.1,0) {};
				\node at (-5.9,-5) {};
				
				\node (v3) at (-3.5,0) {};
				\node (v4) at (0.5,0) {};
				\node at (-2.9,5) {};
				\node at (-2.9,-5) {};
				\draw  (v2) edge node[auto] {e} (v3);
				\node (v5) at (2.1,0) {};
				\node (v8) at (3,0) {};
				\draw  (v4) edge (v5);
				\draw  (-1.5,0) ellipse (2 and 5);
				\draw  (3,-1) rectangle (2,1);
				\node (v9) at (5,0) {};
				
				\draw  (v8) edge (v9);
				\draw[thick, dashed,] (-2,0) -- (-3.5,0);
				\draw[thick, dashed,] (-2,0) -- (0,-3);
				\draw[thick, dashed,] (0.4,1.5) -- (-3,3);
				\draw[thick, dashed,] (-2,0) -- (0.4,1.5);
				\draw[thick, dashed,] (-2,0) -- (-2,-3);
				\draw[thick, dashed,] (-2,0) -- (-3,3);
				\draw[thick, dashed,] (-3,-3) -- (0,-3);
				\draw[thick, dashed,] (2,-0.5) -- (2.4,0.3) -- (2.7,-0.4) -- (3,0);
				\draw[thick, dashed,] (2,0.5) -- (2.4,0.3);
				\draw[thick, dashed,] (2,-0.5) -- (2.7,-0.4);
			\end{tikzpicture}
		}
		\caption{1 vertical\\1 ship}
		\label{fig:lem2}
	\end{subfigure}
	\begin{subfigure}[t]{0.28\textwidth}
		\centering
		\resizebox{\textwidth}{130pt}
		{
			\begin{tikzpicture}[scale=0.4]

				\draw  (-5,5) rectangle (5,-5);
				
				\draw  (-1.5,0) ellipse (3.5 and 5);

				\draw[thick, dashed,] (-2,1) -- (-5,0);
				\draw[thick, dashed,] (-2,1) -- (1,-3);
				\draw[thick, dashed,] (0.4,1.5) -- (-3,3);
				\draw[thick, dashed,] (-2,1) -- (0.4,1.5);
				\draw[thick, dashed,] (-2,1) -- (-4,-2);
				\draw[thick, dashed,] (-2,1) -- (-3,3);
				\draw[thick, dashed,] (-4,-2) -- (1,-3);
				\draw[thick, dashed,] (-2,1) -- (2,0);
				\draw(2,0) edge node[auto]{e} (5,0);
				
			\end{tikzpicture}
		}
		\caption{1 vertical}
		\label{fig:lem3}
	\end{subfigure}

\end{figure}

It means that there are several ``vertical'' components that wrap around the torus, several ``ships'' --- planar components, and they are connected by the edges like $e$ that cannot be removed without a loss of cellularity.

Note that each ``vertical'' component may have at least three vertices as it should wrap around the torus and $d<\frac 12$, and therefore at least four, as there are no multiple edges and all degrees are at least 3. So in the case of at least two different ``vertical'' components there should be at least 8 vertices, but we have only 7. If there is a ``ship'', then it has to contain at least 4 (even more, if we look more accurately) vertices since otherwise the third vertex has degree at most two. Since the embedding is cellular, there exists a ``vertical'' component, and thus we again have at least 8 vertices, which contradicts the statement. So there are no ``ships'' and in this direction there is at most one such edge (in a sense that if we cut the torus across this edge and get a cylinder then no other edge can be removed to make the embedding in the cylinder non-cellular). Similarly, inside the only ``vertical'' component, there is also at most one edge in the other direction, after removing  which the embedding becomes non-cellular.

If $N=6$, then we have $n_6 = 0$ and $n_3 \leq 3$ as $n_3\leq n_5$ and $n_3+n_5\leq6$. So we have at most $3n_3\leq 9$ edges prohibited to remove by the condition on degrees and at most 2 by the one on cellularity. So, among at least $2N=12$ edges, only $9+2=11$ are disallowed to be removed.

If $N=7$ and $n_6 >= 2$ then, by the remark on the structure of vertices of the degree 6 the edge between these 2 vertices is a diagonal of a rhombus and can safely be removed. If $n_6=1$ then we have $n_3+n_5\leq 6$ (the sum is less than 7 and must be even) and $n_3\leq n_5+2n_6=n_5+2$. Thus $n_3\leq 4$. Our graph has at least $2N=14$ edges, so in the case $n_3\leq 3$ we have at most $3n_3+2\leq11$ prohibited edges and we are done. The last case is $n_3 = 4$.

If $n_6=1$ and $n_3=4$, then $n_5=2$ and these 2 vertices of degree 5 share a common adjacent vertex of degree 3, and all these 3 vertices are adjacent to the one of the degree 6. That gives us that there is a rhombus $abcd$, where $\deg a = 5$, $\deg b = 6$, $\deg c = 5$ and $\deg d = 3$, from which we can safely remove the diagonal edge $bd$. That finishes the proof.
\end{proof}

\section{Proof of Theorem \ref{thm6}}
\subsection{First step: finding all possible subgraphs}

To prove the desired Theorem \ref{thm6}, we ask {\ttfamily surftri} to enumerate all cellular toroidal graphs on 6 vertices having $6\cdot2-1=11$ edges and all vertices of degree no less than 3 (these are the subgraphs described in lemma \ref{lm}). There are 252 such graphs. Then for each graph we find two shortest homologically independent cycles and write down all the possible systems (\ref{syst}) from page \pageref{syst} (the possibilities are described in Lemmas \ref{lc1} -- \ref{lc4}) and try to isolate its solution (or solutions) on the set
$$ \begin{array}{ll}
	 |{v_i}_x| \leq d, & i = 1..E \\
	 |{v_i}_y| \leq d, & i = 1..E \\
	 d \in [d_0, d_1] & \\
	 \dist(p_i, p_j) \geq d, &
	 \end{array}
$$
 where $d_0$ and $d_1$ are the known bounds on $d$ (the lower one comes from the ``guess'' for the optimal configuration, the upper one is theoretical). The isolation is performed using the branch and bounds method, i.e., on each iteration we subdivide the allowed intervals for each variable and check the possibility of the feasibility on the set using minimal and maximal estimates. We've used two different methods for choosing the independent variables to be localized by branching. 

In the first approach, to minimize the number of independent variables we first solve the linear part of the system concerning only $x$ coordinates, where the space of the solutions has the dimension $E-F=N$, and we put bounds only on the $N$ coordinates ${v_j}_x, j\in J$ that determine the solution (also, we put bounds on $d$), while the $y$-coordinates are determined from the equations ${v_i}_y = \pm \sqrt{d^2-{v_i}_x^2}$. The feasibility is checked in the following steps:
\begin{enumerate}
	\item Calculate the bounds on the ${v_i}_x$ for $i \not\in J$ from the linear system on $x$-s;
	\item Calculate the bounds on the $|{v_i}_y|$ corresponding to the already bounded $x$-s from the $|v_i| = d$ equations;
	\item Enumerate all possible $2^{E}$ signs for ${v_i}_y$;
	\item Minimax check of the feasibility of the linear system on ${v_i}_y$.
\end{enumerate}

In the second approach, we solve the linear part on both $x$-s and $y$-s (i.e., on the vectors), and then use the expressions 
$$ \begin{array}{l}
{v_i}_x=d\cdot\cos(\alpha_i),\\
{v_i}_y=d\cdot\sin(\alpha_i)\\
\end{array}
$$
for the base vectors. And we put bounds on these $\alpha_i$. Here the process of checking the feasibility looks like this:
\begin{enumerate}
	\item Calculate the bounds on the ${v_i}_x$ and ${v_i}_y$ for $i \not\in J$ from the linear system;
	\item Check lengths of these vectors.
\end{enumerate}

{\noindent \bf Remark}. Actually, the second approach appears to work faster in general, since there is no need to enumerate all the possible combinations of signs for $y$-s, but, firstly, the problem for $N=6$ has been solved by using only the first approach, and, secondly, for a few graphs on $7$ vertices it also appeared to be faster. And the preliminary elimination for $N=7$ has also been done using the first approach.

If the solution is possible and the desired precision is not achieved, we continue the subdivision.

In such a way, we find (subgraphs of) locally optimal configurations without isolated vertices with an arbitrary precision. Then, we do the same procedure for the graphs having less vertices and try to find some place for an isolated vertex there.

This gives us some set of embeddings, but we are still unable to distinguish the subgraphs of the optimal configuration. So, at this stage the computation shows that the following theorem holds:

\begin{thm} Each irreducible optimal configuration on the $N=6$ vertices has a subgraph isometric to one of the graphs in the Figures \ref{fig:left6_1} and \ref{fig:left6_2}, where each red square (that hardly could be seen on the image) determines the possible region for the point, and green circles are of radius about $d$. The precise numerical bounds could be found in the Appendix \ref{app_graphs6}.

\begin{figure}[htb]%
	 \centering
        \begin{subfigure}[b]{0.4\textwidth}
                \centering
                \fbox{\includegraphics[clip, trim = 1 1 3 3,scale=0.3]{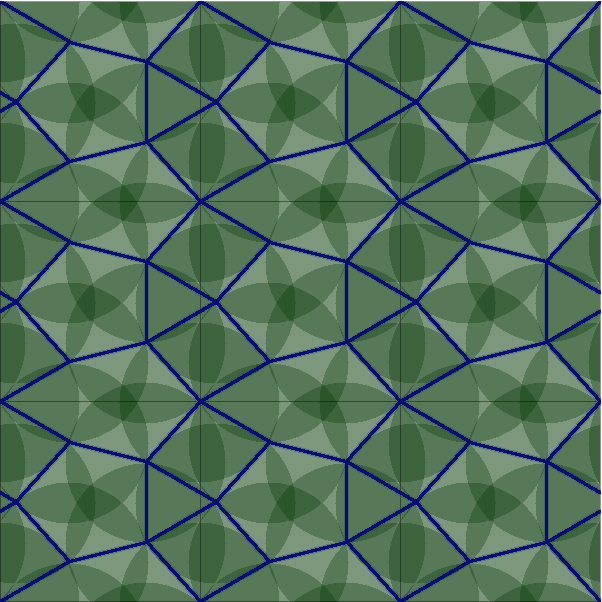}}
                \caption{$G_1$}
        \end{subfigure}%
        \begin{subfigure}[b]{0.4\textwidth}
                \centering
                \fbox{\includegraphics[clip, trim = 1 1 3 3,scale=0.3]{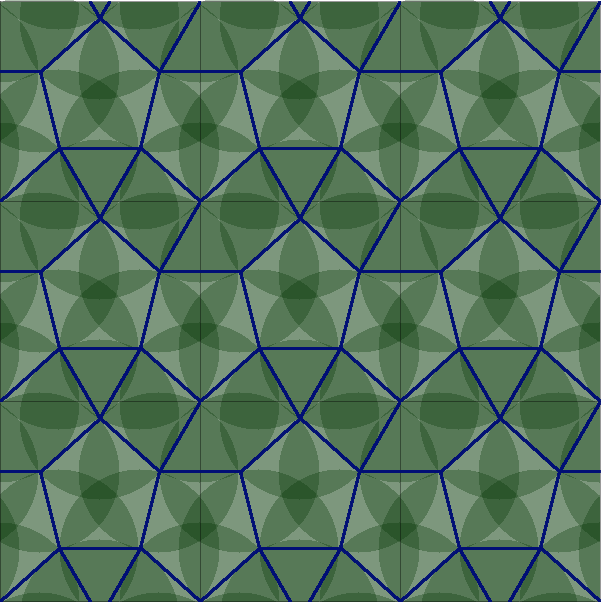}}
                \caption{$G_2$}
        \end{subfigure}%
    \caption{Graphs pretending to be the subgraphs of the optimal}
    \label{fig:left6_1}
\end{figure}
        
\begin{figure}[htb]%
	\centering
        \begin{subfigure}[b]{0.4\textwidth}
                \centering
                \fbox{\includegraphics[clip, trim = 1 1 3 3,scale=0.3]{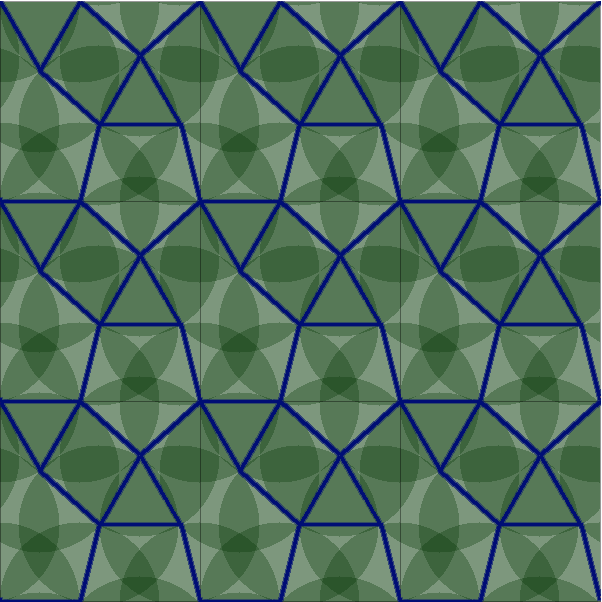}}
                \caption{$G_3$}
        \end{subfigure}%
        \begin{subfigure}[b]{0.4\textwidth}
                \centering
                \fbox{\includegraphics[clip, trim = 1 1 3 3,scale=0.3]{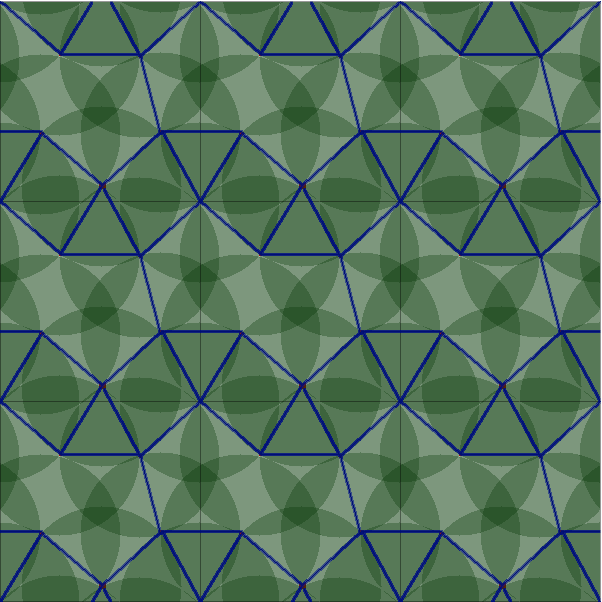}}
                \caption{$G_4$}
        \end{subfigure}%
    \caption{Graphs pretending to be the subgraphs of the optimal}
    \label{fig:left6_2}
\end{figure}
\end{thm}

\subsection{Second step: distinguishing subgraphs}
\label{disting}

To show that the only reason these graphs are found by the code is because they are the subgraphs of the optimal one showed on Figure \ref{fig:opt6} and do not lead to another optimal configuration, we distinguish two possibilities: if there surely is an obtuse angle in the graph, then the corresponding vertex should either be isolated (it is out of interest for us as in this case we will get this graph from the enumeration of graphs with less vertices), or there should be another edge from it. Looking at Figure \ref{fig:left6_2}, we easily see that for $G_3$ and $G_4$ there are obtuse angles, allowing the only possibility of adding an edge. If the edge is added, we obtain the graph shown in Figure \ref{fig:opt6}, where we can delete another edge from, and reduce the problem to, for example, the graph $G_1$. Now we are left with the two graphs $G_1$ and $G_2$. To deal with them (as the pentagons are hard objects for manual maximization) we use the following techniques: we will prove that within the prescribed rectangles there is at most one solution of System (\ref{syst}), and since the one in Figure \ref{fig:opt6} satisfies these bounds, the solutions coincide. For convinience, we replace $d^2$ by just $2d$, changing the bound in the appropriate way.

To prove the propsition, first of all we rewrite System (\ref{syst}) in coordinates of the points of the configuration (we've multiplied by $\frac 12$ for convinience):

\begin{equation}
\label{syco}
\left\{
\begin{array}{l}
	x_0 = 0\\
	y_0 = 0\\
	\frac 12 (x_i - x_j + \delta_{ij})^2 + \frac 12 (y_i - y_j + \gamma_{ij})^2 - d = 0, \enskip 0\leq i < j < N, (ij)\text{ is an edge}
\end{array} 
\right.
\end{equation}
where $\delta_{ij}, \gamma_{ij} \in \{0, \pm1\}$ shows which distance (of the nine possible) between the points we have to consider. We assume that the precision is good enough to determine these uniquely.

Suppose we have some $$X_0 = (x_1^0, y_1^0, x_2^0, y_2^0, \ldots, x_{N-1}^0, y_{N-1}^0, d^0)$$ such that $F(X_0) = 0$, where $F$ is the function $\mathbb R^{2N-1} \to \mathbb R^{2N-1}$ on the left hand side of System (\ref{syco}) (we've substituted $x_0=y_0=0$). Since $F$ is quadratic and $F(X_0)=0$, we have $$F(X_0+\eps v) = \eps \dd F_{X_0}(v) + \eps^2 \frac 12 \dd^2F_{X_0}(v,v)$$ for any vector $$v = (dx_1, dy_1, dx_2, dy_2, \ldots, dx_{N-1}, dy_{N-1}, dd)\in R^{2N-1}.$$ We claim that for $\eps>0$ small enough\footnote{Recall that $\|u\|_\infty$ is the maximal of the absolute values of the coordinates of $u$.} $$\|F(X_0+\eps v)\|_\infty>0$$ for any unit (in $\|\cdot\|_\infty$ norm) vector $v$, and if we would have an explicit bound on $\eps$ then we could establish the uniqueness.

Firstly\footnote{$F_{(ij)}$ means the coordinate that corresponds to the equation for the edge $ij$.}, we compute $$(d^2F(v,v))_{(ij)} = (dx_i - dx_j)^2 + (dy_i - dy_j)^2$$ and thus $$\left\|\frac 12 d^2F(v,v)\right\|_\infty \leq 4.$$ Now, we have to estimate $\|\dd F(v)\|_\infty$ from below.

We have $$(\dd F_{X_0}(v))_{(ij)} = (x_i^0 - x_j^0 + \delta_{ij})(dx_i - dx_j) + (y_i^0 - y_j^0 + \gamma_{ij})(dy_i - dy_j) - dd,$$
where the coefficients are actually the coordinates of the edge vectors. We are interested in
\begin{equation}
\inf\limits_{\|v\|_\infty=1}\|\dd F(v)\|_\infty = \dfrac{1}{\|\dd F^{-1}_{X_0}\|_\infty}.
\label{infin}
\end{equation}

However, we do not know the coefficients exactly, so we need the following well-known lemma from the theory of linear operators:
\begin{lemma}
For any two matrices $A$ and $B$ with $A-B=\Delta$, the following inequality holds $$\|A^{-1}-B^{-1}\| \leq \dfrac{\|B^{-1}\|^2\|\Delta\|}{1-\|\Delta\|\|B^{-1}\|},$$ if the denominator is positive.
\end{lemma}

Applying this for $A = \dd F_{X_0}$ and $B = \dd F_{\tilde X}$ where $\tilde X$ is some point configuration, whose edges lie within the $\delta$-neighborhood of that of $X_0$, we (noticing that since matrices $A$ and $B$ are different by $\delta$ only in 4 positions in each row) get $\|\Delta\|_\infty=\|A-B\|_\infty \leq 4\delta$ and hence
$$\|\dd F_{X_0}^{-1}\|_\infty \leq \|B^{-1}\|_\infty + \dfrac{4\delta\|B^{-1}\|_\infty^2}{1-4\delta\|B^{-1}\|_\infty}.$$

Substituting this into (\ref{infin}) we get the lower bound on $\dd F_{X_0}(v)$; denote it by $h$. So, we have $$\|F(X_0+\eps v)\|\geq \eps h - 4\eps^2.$$ which is positive provided $\eps < \frac {h}{4}$. So, given the bounding intervals for the edges we now can guarantee the uniqueness of the solution by calculating $h$ and checking the inequalities. We do this for the graphs pretending to be optimal and Theorem \ref{thm6} follows.

\section{Proof of theorems \ref{thm7} and \ref{thm8}}

First of all, we notice that the Theorem \ref{thm8} follows from the Theorem \ref{thm7} because we know that $d(8)\leq d(7)$ and there is a configuration with $d=d(7)$. Uniqueness follows from the remark that after removing any vertex we would achieve an optimal configuration on $7$ vertices, all of which are described by the Theorem \ref{thm7}.

Now let's turn to $N=7$. We've used exactly the same method for finding the subgraphs as in the previous section. And (after a couple of weeks of calculations and code optimization) we've been left with the subgraphs of the ones in he figures \ref{fig:opt7_1}--\ref{fig:opt7_3}. Surprisingly, we had not been aware of the graph in the figure \ref{fig:opt7_3}, it has been found by the code. There are a bit more of the found subgraphs than for $N=6$, so we do not provide all the images right here; nevertheless, they can be found at the URL below. There they are categorized by the way used to complete them to optimality and by the step of the elimination process.

Now we have the last problem to distinguish the subgraphs. Note that the graph on Figure \ref{fig:opt7_1} is achieved by enumerating graphs on 6 vertices, and it has $12$ edges, while we enumerate its subgraphs with $2\cdot6-1=11$ edges, so for it we need to prove the existence of the last edge. Similarly, we have to prove that the one last edge should be actually present for the subgraphs of \ref{fig:opt7_3} and the two last edges for $\ref{fig:opt7_2}$. Note that for each edge of $\ref{fig:opt7_2}$, the removal of which doesn't affect the convexity, there is another edge that can be removed simultaneously without the loss of convexity, so we have to deal only with the subgraphs with $13$ edges, whose faces are convex (as otherwise one edge should be added and after that another can be removed [or added up to the optimal graph, if the graph is still ``non-convex''], and we reduce the problem to the one with convex faces by removing an additional edge).

So, now we have to distinguish a ``convex'' subgraph from the optimal graph. We use the similar technique as for $N=6$ involving the differential, but we slightly adjust the method to have better precision. For every possible subgraph embedding $G$ we find the nearest (in the $\|\cdot\|_\infty$-norm on the vertices) graph $\Gamma$ isometric to the graphs on Figures \ref{fig:opt7_1}--\ref{fig:opt7_3} (of course, we remove the floating vertex from \ref{fig:opt7_1} and work on 6 points) and estimate the distance. Then we use the same estimate on the size of the uniqueness domain for that $\Gamma$ (we know its coordinates exactly, so we do not need to estimate the differential in the neighborhood). If the estimate distance is less than this bound, then we are done, having proved that vertices of $G$ coincide with the ones of $\Gamma$.

Also, using the very original problem statement: to maximize $d$ on the set in $\mathbb R^{2N-1}$ described by inequalities
$$\left\{\begin{array}{l}
	0\leq x_i \leq 1\\
	0\leq y_i \leq 1\\
	\dist(p_i, p_j) \geq d,\\
\end{array}\right.$$ we've implemented the Karush--Kuhn--Tucker condition checking for proving that there either is another edge or the graph is not optimal. It could save us a bit of time, but we had already achieved a sufficient precision for the graphs to use the differential estimate by the time these check was written. Actually, we used it only for the one graph on $6$ vertices. And for the another one graph we've had to manually correct the resulting estimates because the basis chosen by the program had not been optimal and the better estimates for the coordinates could be achieved manually.

After performing all these steps, all of the graphs have been proved to coincide with the optimal ones, except for the only one in the following Figure \ref{fig:stb7}. Here the differential of System (\ref{syst}) degrades and the Karush--Kuhn--Tucker conditions are satisfied, so we have to develop an individual strategy for it.

\begin{figure}[h]
        \centering
        \fbox{\includegraphics[clip, trim = 2 1 6 3, scale=0.6]{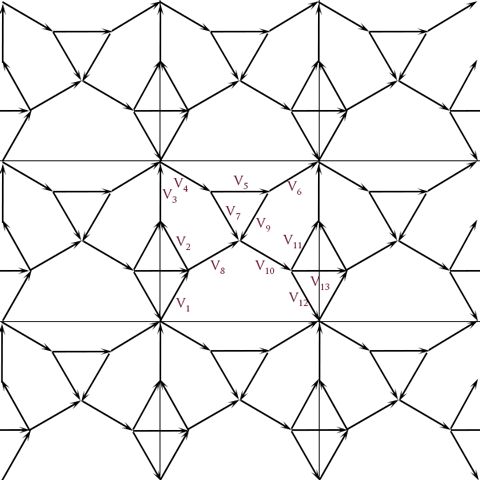}}
        \caption{The most powerful graph}
        \label{fig:stb7}
\end{figure}%

Let's write down the corresponding system of equations:
\begin{equation}
\label{syststb}
\left\{
\begin{array}{l}
{v_1}_x + {v_2}_x + {v_3}_x = 0\\
{v_1}_y + {v_2}_y + {v_3}_y - 1 = 0\\
{v_4}_x + {v_5}_x + {v_6}_x - 1 = 0\\
{v_4}_y + {v_5}_y + {v_6}_y = 0\\
{v_2}_x + {v_3}_x + {v_4}_x + {v_7}_x - {v_8}_x = 0\\
{v_2}_y + {v_3}_y + {v_4}_y + {v_7}_y - {v_8}_y = 0\\
{v_5}_x + {v_9}_x - {v_7}_x = 0\\
{v_5}_y + {v_9}_y - {v_7}_y = 0\\
{v_6}_x - {v_3}_x - {v_{11}}_x - {v_{10}}_x - {v_9}_x = 0\\
{v_6}_y - {v_3}_y - {v_{11}}_y - {v_{10}}_y - {v_9}_y = 0\\
{v_{13}}_x - {v_{11}}_x + {v_2}_x = 0\\
{v_{13}}_y - {v_{11}}_y + {v_2}_y = 0\\
{v_{12}}_x + {v_1}_x - {v_{13}}_x = 0\\
{v_{12}}_y + {v_1}_y - {v_{13}}_y = 0\\
{v_i}_x^2+{v_i}_y^2 - d^2 = 0~~~~~~~~~~~~~~~~1\leq i \leq 13
\end{array}
\right.
\end{equation}

It is a polynomial system of 27 equations in 27 variables. As it was once noticed above, there are known methods for isolating the roots of the polynomial systems. However, this system also appears to have an infinite number of complex solutions, so these methods fail, if applied directly. So we try to do some algebraic localization to eliminate the infinite family of solutions. Adding the condition $d \ne 0$, rewritten as $d\cdot s = 1$ by adding an independent variable seems to do the thing: now the {\ttfamily Maple}'s {\ttfamily Isolate} function from {\ttfamily RootFinding} isolates the $203$ roots of the modified system. Now, we check that there is a unique root of the system even on the set $0 \leq d \leq 0.5$, and thus it is the one we know, which coincides with the optimal. That finishes the proof of the Theorems \ref{thm7} and \ref{thm8}.

{\noindent \bf Remark}. One may think that adding such condition $d\cdot s = 1$ could do the thing for all the graphs; however, that is far from the general case.

\section{Conclusion}

For $N=9$, we present a result of numerical simulations.

The code used for solving the problem and its output for $N=6$ and $N=7$ is available at \url{https://dl.dropbox.com/u/124992374/index.xhtml}. Note that it uses some randomization and has some asserts that may fail sometimes. If the case, the program just needs to be restarted for the failed graph.

\medskip

\noindent{\bf Acknowledgment.} We wish to thank Alexey Tarasov, Vladislav Volkov and Brittany Fasy for some useful comments and remarks, and especially Thom Sulanke for modifying {\ttfamily surftri} to suit our purposes.

\newpage
\appendix
\section{Numerical bounds for the graphs}
\label{app_graphs6}
Here we present the numerical bounds for the edges of $G_1$ and $G_2$ produced by the code

~

\noindent\begin{minipage}[t]{0.48\textwidth}
$G_1$:
\begin{alltt}{\sffamily\small{ab: (0.346659 0.34685, -0.200409 -0.200038)
ac: (0.346668 0.34686, 0.200022 0.200375)
ad: (-0.265502 -0.265339, 0.299724 0.299871)
ae: (-0.265475 -0.265366, -0.29985 -0.299742)
bf: (-0.265466 -0.265357, -0.29988 -0.299747)
be: (0.387784 0.387866, -0.0997594 -0.0994001)
cd: (0.38772 0.387911, 0.099223 0.100013)
cf: (-0.265475 -0.265339, 0.299738 0.299875)
df: (0.346668 0.34686, 0.200036 0.200367)
de: (-0.000109204 0.000109204, 0.4004 0.400427)
ef: (0.346696 0.346832, -0.200325 -0.200083)}}
\end{alltt}
\end{minipage}
\begin{minipage}[t]{0.48\textwidth}
$G_2$:
\begin{alltt}{\sffamily\small{ab: (0.299465 0.30012, 0.265043 0.265794)
ac: (-0.300166 -0.299428, 0.265001 0.265842)
ad: (-0.200635 -0.199762, -0.347011 -0.346512)
be: (0.200062 0.200362, -0.346848 -0.346663)
bc: (0.400342 0.400452, -0.00759171 0.00810797)
bf: (-0.0997579 -0.0994303, 0.387787 0.387861)
cd: (0.0992846 0.0998033, 0.387762 0.387948)
ce: (-0.200335 -0.200062, -0.346848 -0.346687)
de: (-0.299865 -0.299592, 0.265332 0.265718)
df: (0.400397 0.400451, -0.00456848 0.00469225)
ef: (-0.299956 -0.299738, -0.265484 -0.265304)}}
\end{alltt}
\end{minipage}
 
~
 
For $G_1$ we have $\delta<0.002$ and $\eps<0.002$. After calculating the inverse of the differential we get $h\geq \frac{1}{16}$, and so the sufficient inequality $\eps<\frac h4$ holds.

Similarly, for $G_2$ we get $h\geq \frac{1}{35}$ and still $\eps < \frac {h}{4}$.

\newpage
\section{Coordinates of the graphs}
\label{coords}
\renewcommand{\frac}{\dfrac}

Here are some exact values for the coordinates of the vertices of the optimal graphs:

$N=2: \left(0, 0\right), \left(\frac 12, \frac 12\right).$

$N=3: \left(0, 0\right), \left(\frac 12, \frac{\sqrt{3}}{2}\right), \left(\frac{\sqrt{3}-1}{2}, \frac{\sqrt{3}-1}{2}\right).$

$N=4: \left(0, 0\right), \left(\frac 12, \frac{\sqrt{3}}{2}\right), \left(\frac{\sqrt{3}-1}{2}, \frac{\sqrt{3}-1}{2}\right), \left(\frac{\sqrt{3}}{2}, \frac 12\right).$

$N=5: \left(0, 0\right), \left(\frac 25, \frac 15\right), \left(\frac 45, \frac 25\right), \left(\frac 15, \frac 35\right), \left(\frac 35, \frac 45\right).$

The following was achieved by writing down the system of equations for the graph as it appeared to be (i.e., the edges that seem to be horizontal were assumed horizontal, etc) and solving it in {\ttfamily MATLAB}.
\setlength{\mathindent}{0pt}
\begin{equation*}
\begin{array}{l}
N=6: (0, 0),\\
\left(\frac 14\,\sqrt {3}-\frac {1}{12}\,\sqrt {2}\sqrt {3\,\sqrt {3}+2}+\frac {1}{12},
-\frac {1}{12}\, \left( -3\,\sqrt {3}+\sqrt {2}\sqrt {3\,\sqrt {3}+2}-1 \right) \sqrt {3}\right),\\
\left({\frac {11}{12}}-\frac 14\,\sqrt {3}+\frac {1}{12}\,\sqrt {2}\sqrt {3\,\sqrt {3}+2},
-\frac {1}{12}\, \left( -3\,\sqrt {3}+\sqrt {2}\sqrt {3\,\sqrt {3}+2}-1 \right) \sqrt {3}\right),\\
\left(\frac 12,
1-\frac {1}{12}\,\sqrt {3}\sqrt {2}\sqrt {3\,\sqrt {3}+2}+\frac 13\,\sqrt {3}-\frac 14\,\sqrt {2}\sqrt {3\,\sqrt {3}+2}\right),\\
\left({\frac {5}{12}}-\frac 14\,\sqrt {3}+\frac {1}{12}\,\sqrt {2}\sqrt {3\,\sqrt {3}+2},
\frac 54+\frac 14\,\sqrt {3}-\frac 14\,\sqrt {2}\sqrt {3\,\sqrt {3}+2}\right),\\
\left({\frac {7}{12}}+\frac 14\,\sqrt {3}-\frac {1}{12}\,\sqrt {2}\sqrt {3\,\sqrt {3}+2},
\frac 54+\frac 14\,\sqrt {3}-\frac 14\,\sqrt {2}\sqrt {3\,\sqrt {3}+2}\right).
\end{array}
\end{equation*}

\newcommand{\ddd}{\frac{1}{1+\sqrt{3}}}
For $N=7$ and $8$ for all the optimal graphs we just write down the edge coordinates, as there are only $6$ types of them and they look simpler than the point coordinates:
$\left(\frac 12 \ddd, \pm \frac {\sqrt{3}}{2}\ddd\right), \left(\frac {\sqrt{3}}{2}\ddd, \pm \frac 12 \ddd\right), \left(0, \ddd\right), \left(\ddd, 0\right).$
\begin{thebibliography}{99}
\bibitem{BMP}
Brass P., Moser W. O. J., Pach J.: Research problems in discrete geometry, Springer-Verlag, 2005.

\bibitem{plantri}
Brinkmann G., McKay B. D.: Fast generation of planar graphs, \url{http://cs.anu.edu.au/\textasciitilde bdm/plantri/}

\bibitem{Con1}
Connelly, R.: Juxtapositions rigides de cercles et de spheres. I. Juxtapositions finies. Struct. Topol. (14),
4360 (1988) (Dual French--English text)

\bibitem{Con2}
Connelly, R.: Juxtapositions rigides de cercles et de spheres. II. Juxtapositions infinies de mouvement fini.
Struct. Topol. (16), 5776 (1990) (Dual French--English text)

\bibitem{Con3}
Connelly, R.: Rigidity of packings. Eur. J. Combin. 29(8), 18621871 (2008)

\bibitem{Conw}
Conway J. H., Sloane N. J. A.: Sphere Packings, Lattices and Groups, Second edition, Springer-Verlag, New York, 1993

\bibitem{Dan}
Danzer L.: Finite point-sets on ${\bf S}^2$ with minimum distance
as large as possible, Discr. Math., 60 (1986), 3-66.

\bibitem{Dick}
Dickinson W., Guillot D., Keaton A., Xhumari S.: Optimal packings of up to five equal circles on a square flat torus, Beitr\"age zur Algebra und Geometrie, 52(2) (2011), 315-333.

\bibitem{Fejes}
Fejes T\'oth, L.: Regular figures. International Series of Monographs on Pure and Applied Mathematics,
vol. 48. Macmillan, New York (1964)

\bibitem{FeT}
Fejes T\'oth L.: Lagerungen in der Ebene, auf der Kugel und in
Raum, Springer-Verlag, 1953; Russian translation, Moscow, 1958

\bibitem{HabvdW}
W. Habicht und B.L. van der Waerden, Lagerungen von Punkten auf der Kugel, Math. Ann. {\bf 123} (1951),  223-234.

\bibitem{Mel}
Melissen J. B. M.: Densest packing of six equal circles in a square, Elemente Math. 49 (1994) 27–31.

\bibitem{Mus}
Musin O.~R. and Tarasov A.~S.: The Strong Thirteen Spheres Problem, Discrete \& Comput. Geom., 48 (2012), 128-141.

\bibitem{Sch}
Schaer J.: The densest packing of 9 circles in a square, Canad. Math. Bull. 8 (1965) 273–277.

\bibitem{SchM}
Schaer J., Meir A.: On a geometric extremum problem, Canad. Math. Bull. 8 (1965) 21–27.

\bibitem{SvdW1}
K. Sch\"utte and B.L. v. d. Waerden, Auf welcher Kugel haben
5,6,7,8 oder 9
Punkte mit Mindestabstand 1 Platz? Math. Ann. {\bf 123} (1951),
96-124.

\bibitem{SvdW2}
K. Sch\"utte and B.L. van der Waerden, Das Problem der dreizehn
Kugeln, Math. Ann. {\bf 125} (1953), 325-334.

\bibitem{surftri}
Sulanke T.: Generating irreducible triangulations of surfaces arXiv:math/0606687

\bibitem{Us}
Usikov D. A.: Optimal circle packings of a flat square torus. Private Сommunication, 2004. 
\end{thebibliography}
\end{document}